\documentclass{amsart}

\usepackage{amssymb,amsmath,amsthm,color}

\usepackage[bookmarksopen=true,bookmarksnumbered=true, bookmarksopenlevel=2, colorlinks=true,linkcolor=mblue,
citecolor=mblue,urlcolor=mblue, pdfstartview=FitH]{hyperref}

\usepackage{enumerate}

\usepackage[all]{xy}

\definecolor{mblue}{rgb}{0,0,.8}

\newcommand{\N}{\mathbb N}
\newcommand{\Z}{\mathbb Z}
\newcommand{\Q}{\mathbb Q}

\newcommand\Es[0]{E^\ast}

\renewcommand{\k}{\kappa}

\newcommand{\s}{\sigma}

\newtheorem{thm}{Theorem}[section]
\newtheorem{lem}[thm]{Lemma}
\newtheorem{remark}[thm]{Remark}
\newtheorem{dfn}[thm]{Definition}
\newtheorem{prop}[thm]{Proposition}
\newtheorem{cor}[thm]{Corollary}

\newtheorem{thmx}{Theorem}

\begin{document}

\title[Eisenstein series and overconvergence, II]{Eisenstein series, $p$-adic modular functions, and overconvergence, II}
\author[Ian Kiming, Nadim Rustom]{Ian Kiming, Nadim Rustom}
\address[Ian Kiming]{Department of Mathematical Sciences, University of Copenhagen, Universitetsparken 5, DK-2100 Copenhagen \O ,
Denmark.}
\email{kiming@math.ku.dk}
\address[Nadim Rustom]{Tryg Forsikring A/S, Klausdalsbrovej 601, 2750 Ballerup, Denmark}
\email{restom.nadim@gmail.com}



\begin{abstract} Let $p$ be a prime number. Continuing and extending our previous paper with the same title, we prove explicit rates of overconvergence for modular functions of the form $\frac{\Es_k}{V(\Es_k)}$ where $\Es_k$ is a classical, normalized Eisenstein series on $\Gamma_0(p)$ and $V$ the $p$-adic Frobenius operator. In particular, we extend our previous paper to the primes $2$ and $3$. For these primes our main theorem improves somewhat upon earlier results by Emerton, Buzzard and Kilford, and Roe. We include a detailed discussion of those earlier results as seen from our perspective. We also give some improvements to our earlier paper for primes $p\ge 5$.

Apart from establishing these improvements, our main purpose here is also to show that all of these results can be obtained by a uniform method, i.e., a method where the main points in the argumentation is the same for all primes.

We illustrate the results by some numerical examples.
\end{abstract}

\maketitle

\section{Introduction}\label{sec:intro} Let $p$ be a prime number. By $v_p$ we denote the $p$-adic valuation of $\overline{\Q}_p$, normalized so that $v_p(p)=1$.

Suppose that $k\ge 4$ is an even integer divisible by $p-1$. Then the classical Eisenstein series $E_k$, a modular form of weight $k$ and level $1$ with $q$-expansion
$$
E_k(q) = 1 - \frac{2k}{B_k} \sum_{n=1}^{\infty} \sigma_{k-1}(n) q^n ,
$$
as well as its $p$-deprived counterpart
$$
\Es_k := \frac{E_k - p^{k-1}V(E_k)}{1-p^{k-1}} ,
$$
a modular form of weight $k$ on $\Gamma_0(p)$, are both lifts of a power of the Hasse invariant. Here, we denote as usual by $V$ the Frobenius operator that acts as $q\mapsto q^p$ on $q$-expansions.

This means that the modular function $\frac{\Es_k}{V(\Es_k)}$ is non-vanishing on the ordinary locus of $X_0(p)$ and hence is an overconvergent modular function with some positive rate of overconvergence. These modular functions are specializations of more general functions $\frac{\Es_{\k}}{V(\Es_{\k})}$ for characters $\k$ on $\Z_p^{\times}$ that are trivial on the $(p-1)$st roots of unity. This family of functions plays a decisive role in Coleman's original proof in \cite{coleman_banach} of the existence of $p$-adic families (``Coleman families") of modular forms passing through a given eigenform of finite $p$-slope. Although more conceptual approaches to this theory now exist, cf.\ \cite{andreatta_iovita_stevens}, \cite{pilloni}, there are still reasons to continue the study of the family $\frac{\Es_{\k}}{V(\Es_{\k})}$: first, knowledge of rates of overconvergence may be used to compute examples of Coleman families including explicit information about the radius of convergence, cf.\ \cite{cst}, \cite{destefano}; there are precious few such examples in the literature. Secondly, one should remember that the proofs of the ``halo'' conjecture for the primes $2$ and $3$, cf.\ \cite{buzzard_kilford} and \cite{roe_3-adic}, respectively, contain as a decisive element precisely the establishment of explicit rates of overconvergence of the family $\frac{\Es_{\k}}{V(\Es_{\k})}$. Probably for these reasons, Coleman formulated and proved for $p=2,3$ in \cite{coleman_eisenstein} a general conjecture about this, a version of which was proved for $p\ge 5$ in \cite{akw}. It can be noted that in both cases, i.e., for $p=2,3$ as well as for $p\ge 5$, a rate of overconvergence of the family $\frac{\Es_{\k}}{V(\Es_{\k})}$ is ultimately derived from information about rates of overconvergence of specializations $\frac{\Es_k}{V(\Es_k)}$ with $k$ as above. For $p=2$, $p=3$, and $p\ge 5$ this information comes from \cite{emerton_thesis} and \cite{buzzard_kilford}, from \cite{roe_3-adic}, and from \cite{kr}, respectively.

With motivation as just described, the purpose of this paper is two-fold: first, to present improved statements about the rates of overconvergence of modular functions $\frac{\Es_k}{V(\Es_k)}$ for all primes, and secondly also to show that all of these statements can be obtained by a uniform method that works for all primes.

We prove the following theorem.

\begin{thmx}\label{thm:main} Let $p$ be a prime, and let $k\in\N$ be even with $k\ge 4$ and $k\equiv 0 \pmod{p-1}$. Define:
$$
\rho := \left\{ \begin{array}{ll} \frac{1}{p+1} & \mbox{ if } p\ge 5 \mbox{ or } (p\in \{2,3\} \mbox{ and } k\equiv 0 \pmod{2p}) \\ \frac{1}{2p} & \mbox{ otherwise} \end{array} \right.
$$
and
$$
t := \left\{ \begin{array}{ll} 3 & \mbox{ if } p=2, k=4 \\ 3+v_2(k) & \mbox{ if } p=2, k>4 \\ 2+v_p(k) & \mbox{ if } p\ge 3 .\end{array} \right.
$$

Then
$$
\frac{\Es_k}{V(\Es_k)} \in M_0\left(\Z_p, \ge \frac{t}{t+1} \cdot \rho \right).
$$
\end{thmx}

Here, the notation $f\in M_0(\Z_p, \ge \s )$, as will be explained in more detail below in section \ref{sec:prelim}, means that $f$ is a modular function (i.e., weight $0$ modular form) with coefficients in $\Z_p$ that has rate of overcorvergence $r$ whenever $r < \s$.

The theorem improves upon statements from \cite{kr} for $p\ge 5$ by taking $v_p(k)$ into account. The same is true in relation to the papers \cite{buzzard_kilford}, \cite{roe_3-adic} when $p=2,3$, $k\equiv 0 \pmod{2p}$.

However, when $p\in \{2,3\}$ and $k\not\equiv 0 \pmod{2p}$, the theorem is not in optimal shape: in order to keep the proof of the main theorem as simple and straightforward as possible, we avoided making use of special, improved integrality properties for the $U$ operator that are true for $p=2,3$ and rates of overconvergence less than $\frac{1}{2p}$. However, we will explain this in detail in section \ref{subsec:special_p=2,3}. In particular, this will show that central results from the papers \cite{buzzard_kilford}, \cite{roe_3-adic} on the overconvergence of functions $\frac{\Es_k}{V(\Es_k)}$ can also be obtained by the method we present here. The relevant statement is Proposition \ref{prop:special_E/V(E)} below.

We explain in more detail the comparison between Theorem \ref{thm:main} and results from the papers \cite{buzzard_kilford}, \cite{roe_3-adic}, \cite{kr} in section \ref{subsec:numerical} below. We also illustrate with a number of numerical examples.

\section{Preliminaries}\label{sec:prelim} We summarize here the basic facts about overconvergent modular functions that we will need. For a more comprehensive discussion see \cite[Section 2]{kr}.

If $K/\Q_p$ is a finite extension with ring of integers $O$, we have for $r\in O$ the notion of $r$-overconvergent modular functions of level $1$. We denote the $O$-module of these by $M_0(O,r)$. If $p\ge 5$ an element $g\in M_0(O,r)$ has a ``Katz expansion''
$$
g = \sum_{i\ge 0} \frac{b_i}{E_{p-1}^i}
$$
where the coefficients $b_i$ are modular forms that come from certain $O$-modules $B_i(O)$ of modular forms of weight $i(p-1)$. For $i>0$ the modules $B_i(O)$ provide splittings $M_{i(p-1)}(O) = E_{p-1} \cdot M_{(i-1)(p-1)} \oplus B_i(O)$ where $M_l(\cdot)$ denotes modular forms of weight $l$. One puts $B_0(O) := O$. That $g\in M_0(O,r)$ is then equivalent to having $v_p(b_i) \ge i v_p(r)$ for all $i$, and $v_p(b_i) - i v_p(r) \rightarrow \infty$ for $i\rightarrow \infty$. For details on this, confer \cite{katz_padic}, Propositions 2.6.2, 2.8.1, or the review in \cite[Section 2]{kr}.

For $\s \in \Q \cap [0,1]$ we define the $O$-module $M_0(O,\ge \s)$ (in fact an $O$-algebra) as consisting of those overconvergent modular functions $g$ for which $g\in M_0(O,r)$ for some $r\in O$, and such that: whenever $K'/K$ is a finite extension with ring of integers $O'$, and whenever $r'\in O'$ is such that $v_p(r') < \s$ then $g\in M_0(O',r')$.

This may seem like a very convoluted definition, but is in fact very convenient when working with Katz expansions as above: suppose again that $p\ge 5$ and that $g$ has a Katz expansion as above. Then $g\in M_0(O,\ge \s)$ is equivalent to the statement that $v_p(b_i) \ge i \s$ for all $i$, cf.\ \cite[Proposition 2.3]{kr}.

When $p\in \{2,3\}$, due to the simple fact that $E_{p-1}$ then does not exist as a classical modular form, we do not have Katz expansions of overconvergent modular functions as above. Instead, one option would be to base oneself on Vonk's paper \cite{vonk_small_primes} where it is shown that one can use inverse powers of $E_4$ and $E_6$, respectively, to obtain orthonormal bases. In this paper, though, we will be making use of the orthornormal bases given by Loeffler in \cite{loeffler}: for $p=2,3,5,7,13$ the function
$$
f_p(z) := \left( \frac{\eta(pz)}{\eta(z)} \right)^{\frac{24}{p-1}},
$$
where $\eta$ the usual Dedekind $\eta$-function, is a uniformizer for the genus $0$ modular curve $X_0(p)$.

Given an element $g\in M_0(O,r)$ one can compute an expansion, formal to begin with,
$$
g = \sum_{i=0}^{\infty} a_i f_p^i
$$
where $a_i\in O$. From a practical point of view, computing this expansion from the $q$-expansion of $g$ is easily done recursively due to the fact that the $q$-expansion of $f_p$ starts with $q$. Given this expansion, \cite[Corollary 2.2]{loeffler} implies that for $v_p(r) < \frac{p}{p+1}$ we have $g\in M_0(O,r)$ if and only if $v_p(a_i) \ge \frac{12}{p-1} v_p(r) i$ for all $i$ and $v_p(a_i) - \frac{12}{p-1} v_p(r) i \rightarrow \infty$ for $i\rightarrow \infty$. These facts imply the following statement that we single out as a proposition for reference in arguments below.

\begin{prop}\label{prop:oc_vs_loeffler} Suppose that $p\in \{2,3,5,7,13\}$, that $g\in M_0(O,r)$ for some $r$, and that $g = \sum_{i=0}^{\infty} a_i f_p^i$ with $a_i\in O$, formally in $q$-expansions. Then for $\s \le \frac{p}{p+1}$ we have $g\in M_0(O,\ge \s)$ if and only if
$$
v_p(a_i) \ge \frac{12}{p-1} \s i
$$
for all $i$.
\end{prop}

We will also make use of the following observation that is proved in \cite{kr}.

\begin{prop}\label{prop:conv_in_qexp} (Cf.\ \cite[Corollary 2.7]{kr}.) Suppose that $g_i\in M_0(O, \ge \s)$ is a sequence of elements that converges in $q$-expansion to an element $g(q) \in O[[q]]$. Then $g(q)$ is the $q$-expansion of an element $g\in M_0(O,\ge \s)$.
\end{prop}

In the paper \cite{kr} the standing assumption is $p\ge 5$, but the proof of Corollary 2.7 of that paper is easily adapted to the cases $p=2,3$, for instance by using Loeffler's orthonormal bases above instead of working with Katz expansions.

Finally, the $U$ operator will obviously play a major role in the proofs. Of particular importance to us is its interplay with integral structures. We have the following general statement, a consequence of well-known properties of $U$, going all the way back to the Integrality Lemma 3.11.4 in Katz' foundational paper \cite{katz_padic}.

\begin{prop}\label{prop:U_integrality} Suppose that $\s \le \frac{1}{p+1}$. Then the $U$ operator maps $M_0(O, \ge \s)$ to $\frac{1}{p} \cdot M_0(O,\ge p\s )$.
\end{prop}

\subsection{Special preliminaries for \texorpdfstring{$p=2,3$}{p=2,3}} For the primes $p=2,3$, our work will take as a starting point knowledge about rates of overconvergence of the modular functions $\frac{E_4}{V(E_4)}$ and $\frac{E_6}{V(E_6)}$. This is readily obtained: the following formulas are easily found and proved via comparison of $q$-expansions, the first and last of these are also given in \cite{calegari_proper}, cf.\ the remark on p.\ 212:

For $p=2$ we have
$$
\frac{E_4}{V(E_4)} = \frac{1+2^8 f_2}{1+2^4 f_2} , \quad \frac{E_6}{V(E_6)} = \frac{1-2^9 f_2}{1-2^3 f_2} ,
$$
and for $p=3$:
$$
\frac{E_4}{V(E_4)} = \frac{1+3^5 f_3}{1+3 f_3} , \quad \frac{E_6}{V(E_6)} = \frac{1-2\cdot 3^5 f_3-3^9 f_3^2}{1+2\cdot 3^2 f_3 - 3^3 f_3^2} .
$$

We determine the rate of overconvergence of the above modular functions via Proposition \ref{prop:oc_vs_loeffler}: we can turn the expression of the given modular function $h$ as a rational function in $f_p$ into a formal power series $h = 1 + a_1 f_p + a_2 f_p^2 + \ldots $ where the coefficients for the cases above will be in $\Z_p$. Then $h\in M_0(\Z_p, \ge \s)$ if we have $v_p(a_i) \ge \frac{12}{p-1} \s i$ for all $i$. If we also have $v_p(b_i) \ge \frac{12}{p-1} \s i$ for all $i$ for the coefficients in the expansion $1/h = 1 + b_1 f_p + b_2 f_p^2 + \ldots$, we can conclude that $h$ is even a unit in $h\in M_0(\Z_p, \ge \s)$.

In this way, the formulas above can easily be seen to imply the following proposition.

\begin{prop}\label{prop:rate_of_oc_E/VE} For $p=2$ the functions $E_4/V(E_4)$ and $E_6/V(E_6)$ are units in $M_0(\Z_p,\ge \frac{1}{p+1})$ and $M_0(\Z_p, \ge \frac{1}{2p})$, respectively.

For $p=3$ the functions $E_4/V(E_4)$ and $E_6/V(E_6)$ are units in $M_0(\Z_p, \ge \frac{1}{2p})$ and $M_0(\Z_p,\ge \frac{1}{p+1})$, respectively.
\end{prop}

\section{Proofs}

\subsection{Proof of Theorem \ref{thm:main}}\label{subsec:proof_thm_A} We start the proof by defining an auxiliary modular form $F$ as well as certain quantities attached to it.

\begin{dfn}\label{dfn:F_etc} Let $p$ be a prime, and let $k\in\N$ be even with $k\ge 4$ and $k\equiv 0 \pmod{p-1}$. Define $\rho$ and $t$ as in Theorem \ref{thm:main}, i.e.,
$$
\rho := \left\{ \begin{array}{ll} \frac{1}{p+1} & \mbox{ if } p\ge 5 \mbox{ or } (p\in \{2,3\} \mbox{ and } k\equiv 0 \pmod{2p}) \\ \frac{1}{2p} & \mbox{ otherwise,} \end{array} \right.
$$
$$
t := \left\{ \begin{array}{ll} 3 & \mbox{ if } p=2, k=4 \\ 3+v_2(k) & \mbox{ if } p=2, k>4 \\ 2+v_p(k) & \mbox{ if } p\ge 3 ,\end{array} \right.
$$
and define additionally,
$$
F := \left\{ \begin{array}{ll} E_4^{k/4} & \mbox{ if } p=2 \mbox{ and } k \equiv 0 \pmod{4} \\ E_4^{(k-6)/4} E_6 & \mbox{ if } p=2 \mbox{ and } k\equiv 2 \pmod{4} \\ E_6^{k/6} & \mbox{ if } p=3 \mbox{ and } k \equiv 0 \pmod{6} \\ E_8 E_6^{(k-8)/6} & \mbox{ if } p=3 \mbox{ and } k \equiv 2 \pmod{6} \\ E_4 E_6^{(k-4)/6} & \mbox{ if } p=3 \mbox{ and } k \equiv 4 \pmod{6} \\ E_{p-1}^{k/(p-1)} & \mbox{ if } p \ge 5 \end{array} \right.
$$
\end{dfn}

The following corollary follows immediately from Proposition \ref{prop:rate_of_oc_E/VE}.

\begin{cor}\label{cor:oc_F/V(F)} With notation as in Definition \ref{dfn:F_etc} we have that $\frac{F}{V(F)}$ is an overconvergent modular function and is in fact invertible in $M_0(\Z_p,\ge \rho )$.
\end{cor}

\begin{proof} The form $F$ is in all cases a lift of a power of the Hasse invariant, and so the function $\frac{F}{V(F)}$ is defined on the ordinary locus and is thus overconvergent for some positive rate of overconvergence. That it is in fact a unit of $M_0(\Z_p,\ge \rho )$ follows for $p=2,3$ immediately from Proposition \ref{prop:rate_of_oc_E/VE} using the fact the $V$ is multiplicative. For $p\ge 5$ the statement follows from the ``Coleman--Wan theorem'', \cite[Lemma 2.1]{wan}.
\end{proof}

\begin{cor}\label{cor:U(F)/F_oc} Let $F$, $\rho$ and $t$ be as in Definition \ref{dfn:F_etc}. Then
$$
\frac{U(F)}{F} \in \frac{1}{p} \cdot M_0(\Z_p,\ge p\rho ) .
$$
\end{cor}

\begin{proof} As $\frac{F}{V(F)} \in M_0(\Z_p, \ge \rho)$ by Corollary \ref{cor:oc_F/V(F)}, we find by using the integrality properties of $U$, cf.\ Proposition \ref{prop:U_integrality}, that
$$
\frac{U(F)}{F} = U \left( \frac{F}{V(F)} \right) \in \frac{1}{p} \cdot M_0(\Z_p, \ge p\rho)
$$
where we used that $\rho \le 1/(p+1)$ as well as ``Coleman's trick", i.e., the identity $U(SV(T)) = U(S)T$.
\end{proof}

The following statement is an extension of a statement used in the proof of \cite[Lemma 3.11]{kr}. The argument is a slight extension of the argument in the proof of \cite[Lemma 3.11]{kr}, so we will be brief.

\begin{prop}\label{prop:eisenstein_congr} With notation as in Definition \ref{dfn:F_etc} we have
$$
\Es_k \equiv E_k \equiv F \pmod{p^t}
$$
in $q$-expansion.

As a consequence we have
$$
\frac{\Es_k}{F} \equiv \frac{U^i(F)}{F} \equiv 1 \pmod{p^t} .
$$
for all $i\in\N$.
\end{prop}

\begin{proof} Notice first that
$$
\Es_k = \frac{E_k - p^{k-1}V(E_k)}{1 - p^{k-1}} \equiv E_k \pmod{p^{k-1}} .
$$

As elementary considerations show that $k-1 \ge 3+v_p(k)$ holds except when $p=2$, $k=4$, in which case we have $F= E_4 \equiv \Es_4 \pmod{2^3}$, we see that it suffices to prove the congruence $E_k \equiv F \pmod{p^t}$ in the various cases. These congruences follow readily by considering the explicit formulas for the $q$-expansions of Eisenstein series:
$$
E_m(q) = 1 - \frac{2m}{B_m} \sum_{i\ge 1} \sigma_{m-1}(i) q^i ,
$$
applying the elementary congruence
$$
(1+ap^s)^n \equiv 1 + nap^s \pmod{p^{1+s+v_p(n)} R}
$$
valid for $s,n\in\N$ and $a$ in any commutative ring $R$, and using the von Staudt--Clausen theorem as well as Fermat's little theorem.

Let us first consider the case $p\ge 5$ and $k\in\N$ divisible by $p-1$. As $B_{p-1}, B_k \equiv -\frac{1}{p} \pmod{\Z_p}$ by the von Staudt--Clausen theorem, we also have $\frac{1}{B_{p-1}}, \frac{1}{B_k} \equiv -p \pmod{p^2 \Z_p}$. The elementary congruence above then shows that
$$
E_{p-1}^{k/(p-1)} \equiv 1 - \frac{2k}{B_{p-1}} \sum_{i\ge 0} \sigma_{p-2}(i) q^i \pmod{p^{2+v_p(k)}} ,
$$
and so the claim follows if we can show that $\frac{2}{B_{p-1}} \sigma_{p-2}(i) \equiv \frac{2}{B_k} \sigma_{k-1}(i) \pmod{p^2}$ for any $i\ge 0$. But that follows again from $\frac{1}{B_{p-1}}, \frac{1}{B_k} \equiv -p \pmod{p^2 \Z_p}$ and from Fermat's little theorem that implies $\sigma_{p-2}(i) \equiv \sigma_{k-1}(i) \pmod{p}$ (as $k$ is divisible by $p-1$.)

The congruences for $p=2,3$ and $k$ divisible by $p$ follow from  $E_4 = 1 + 240 \sum_{i\ge 1} \s_3(i)q^i$ and $E_6 = 1 + 504 \sum_{i\ge 1} \s_5(i) q^i$ in the same way as above.

The congruence for $p=2$ and $k\equiv 2 \pmod{4}$ follows in the same way after noting that $E_4 \equiv 1 \pmod{2^4}$ and that in this case, $3+v_2(k)=4$.

When $p=3$ and $k\equiv 2 \pmod{6}$ one finds $E_8 \equiv E_k \pmod{3^2}$, and the desired congruences follows as $E_6 \equiv 1 \pmod{3^2}$. Finally, when $p=3$ and $k\equiv 4 \pmod{6}$, the desired congruence follows in the same way by replacing $E_8$ by $E_4$.

Finally, the second statement of the Proposition follows: as the $U$ operator preserves congruences in $q$-expansions, and as $U(\Es_k) = \Es_k$, we have
$$
U(F) \equiv U(\Es_k) = \Es_k \equiv F \pmod{p^t},
$$
and so $U^i(F) \equiv \Es_k \equiv F \pmod{p^t}$ for all $i\in\N$ by induction on $i$.
\end{proof}

\begin{remark} Numerical experimentation suggests that the value given for $t$ in Proposition \ref{prop:eisenstein_congr} is optimal in some, but not all cases. It seems that improving the value of $t$ would require other considerations than the ones employed here. We have not pursued this.
\end{remark}

If $p\ge 5$ the statement of the following lemma was already established in the proof of \cite[Lemma 3.12]{kr}, and we can thus limit ourselves to the cases $p=2,3$ in the proof.

\begin{lem}\label{lem:U^i(F)/F-U(F)/F} Let $F$ and $\rho$ be as in Definition \ref{dfn:F_etc}. For $i\in \N$ we have
$$
\frac{U^i(F)}{F} - \frac{U(F)}{F} \in M_0(\Z_p, \ge p\rho) .
$$
\end{lem}

\begin{proof} As in the proof of \cite[Lemma 3.12]{kr} for $p\ge 5$ the point of the proof is to combine known integrality properties for the $U$ operator with the congruence of $\frac{U^i(F)}{F} \equiv 1 \pmod{p^t}$ from Proposition \ref{prop:eisenstein_congr}. As it turns out, we need $t\ge 2$ for $p=3$ (as in the cases $p\ge 5$), but for $p=2$ it is necessary to have $t\ge 3$. We do have that by the Definition \ref{dfn:F_etc}. We assume now that $p\in \{2,3\}$.

The statement is proved by induction on $i$. The induction start $i=1$ being trivial, we assume the statement for some $i\in \N$:
$$
\frac{U^i(F)}{F} \in \frac{U(F)}{F} + M_0(\Z_p, \ge p\rho) .
$$

As $\frac{U(F)}{F} \in \frac{1}{p} \cdot M_0(\Z_p, \ge p\rho)$ by Corollary \ref{cor:U(F)/F_oc}, we then certainly also have $\frac{U^i(F)}{F} \in \frac{1}{p} \cdot M_0(\Z_p, \ge p\rho)$. Expanding the modular function $\frac{U^i(F)}{F} = 1 + \sum_{m\ge 1} a_m f_p^m$, and using Proposition \ref{prop:oc_vs_loeffler} we then know that $v_p(a_m) \ge -1 + \frac{12}{p-1} p\rho m$ for all $m\ge 1$.

We claim that $\frac{U^i(F)}{F} - 1 \in p\cdot M_0(\Z_p, \ge \rho)$:

With $t$ as in Definition \ref{dfn:F_etc} we have by Proposition \ref{prop:eisenstein_congr} that $\frac{U^i(F)}{F} - 1 \equiv 0 \pmod{p^t}$ in $q$-expansion. It follows that $v_p(a_m) \ge t$ for all $m\ge 1$ as is easily seen from the fact that the $q$-expansion of $f_p$ starts with $q$.

Now, to show that $\frac{U^i(F)}{F} - 1 \in p\cdot M_0(\Z_p, \ge \rho)$ we must show that $v_p(a_m) \ge 1 + \frac{12}{p-1} \rho m$ for all $m\ge 1$ (cf.\ again Proposition \ref{prop:oc_vs_loeffler}.) It suffices to show that for all $m\ge 1$ we either have $-1 + \frac{12}{p-1} p\rho m \ge 1 + \frac{12}{p-1} \rho m$ or $t\ge 1 + \frac{12}{p-1} \rho m$. The first of these inequalities holds for $m\ge \frac{1}{6\rho}$ while the second holds for $m\le \frac{(t-1)(p-1)}{12\rho}$. Hence, it is sufficient to have
$$
\frac{(t-1)(p-1)}{12\rho} \ge \frac{1}{6\rho},
$$
i.e., $(t-1)(p-1) \ge 2$. But this is true as $t\ge 3$ when $p=2$, and $t\ge 2$ when $p=3$.

Now, combining the statement $\frac{U^i(F)}{F} - 1 \in p\cdot M_0(\Z_p, \ge \rho)$ with Corollary \ref{cor:oc_F/V(F)} we find that
$$
\frac{U^i(F)}{V(F)} - \frac{F}{V(F)} = \left( \frac{U^i(F)}{F} - 1 \right) \cdot \frac{F}{V(F)} \in p\cdot M_0(\Z_p, \ge \rho) .
$$

Applying $U$ we can then conclude that
$$
\frac{U^{i+1}(F)}{F} - \frac{U(F)}{F} = U \left( \frac{U^i(F)}{V(F)} \right) - U \left( \frac{F}{V(F)} \right) \in M_0(\Z_p, \ge p\rho)
$$
which completes the induction step.
\end{proof}

\begin{prop} \label{prop:Es_vs_F} Let $k$, $F$, $\rho$, at $t$ be as in Definition \ref{dfn:F_etc}.

Then $\frac{\Es_k}{F} \in \frac{1}{p} M_0(\Z_p,\ge p\rho)$.

We also have that $\frac{\Es_k}{F}$ is invertible in $M_0(\Z_p,\ge \frac{t}{t+1}\cdot p\rho)$.
\end{prop}

\begin{proof} If $p\ge 5$ the first statement was proved in \cite{kr}, cf.\ Theorem C of that paper so we will assume $p\in \{2,3\}$. Since $\Es_k - F$ is seen in all cases to be a cuspidal, Serre $p$-adic modular form of weight $k$, and as $p\le 7$, it follows from \cite[Lemme 3]{serre_zeta} that $U^i (\Es_k - F) \rightarrow 0$ in $q$-expansion as $i\rightarrow \infty$. But, $U(\Es_k) = \Es_k$, and so we have $U^i(F) \rightarrow \Es_k$, and hence also
$$
\frac{U^i(F)}{F} \rightarrow \frac{\Es_k}{F}
$$
in $q$-expansion as $i\rightarrow \infty$. As convergence in $q$-expansion of elements that are in $M_0(\Z_p, \ge p\rho)$ implies that the limit is also in this space, cf.\ Proposition \ref{prop:conv_in_qexp}, we can then conclude from Lemma \ref{lem:U^i(F)/F-U(F)/F} that
$$
\frac{\Es_k}{F} - \frac{U(F)}{F} \in M_0(\Z_p, \ge p\rho) .
$$

That $\frac{\Es_k}{F} \in \frac{1}{p} M_0(\Z_p,\ge p\rho)$ now follows immediately from the first statement of Corollary \ref{cor:U(F)/F_oc}.

For the second statement of the proposition, suppose again first that $p\in \{2,3\}$ and consider the expansion $\frac{\Es_k}{F} = 1 + \sum_{m\ge 1} a_m f_p^m$. By Proposition \ref{prop:oc_vs_loeffler} we now know that $v_p(a_m) \ge -1 + \frac{12}{p-1} \cdot p\rho m$ for all $m\ge 1$. On the other hand, the congruence $\frac{\Es_k}{F} \equiv 1 \pmod{t}$ from Proposition \ref{prop:eisenstein_congr} implies that $v_p(a_m) \ge t$ for all $m\ge 1$. Combining these two estimates, one sees that $v_p(a_m) \ge \frac{12}{p-1} \cdot \frac{t}{t+1} \cdot p\rho m$ for all $m\ge 1$, and this means that
$$
\frac{\Es_k}{F} \in M_0\left(\Z_p,\ge \frac{t}{t+1} \cdot p\rho \right) .
$$

To see that $\frac{\Es_k}{F}$ is in fact invertible in $M_0(\Z_p,\ge \frac{t}{t+1} \cdot p\rho )$, we can borrow an argument from the proof of \cite[Corollary 4.1]{kr}: let $0\le u < \frac{t}{t+1} \cdot p\rho$ be rational and choose a finite extension $K/\Q_p$ with ring of integers $O_K$ such that there are elements $a,r\in O_K$ with
$$
v_p(a) = \frac{6}{p-1} \cdot (\frac{t}{t+1} p\rho - u), \quad v_p(r) = u .
$$

Then $\frac{\Es_k}{F} = 1 + \sum_{m\ge 1} a_m f_p^m = 1 + a \cdot \sum_{m\ge 1} (a_m a^{-1}) f_p^m$, and as we find that
$$
v_p(a_m a^{-1}) - \frac{12}{p-1} v_p(r)m \ge (2m-1) v_p(a)
$$
is positive and goes to $\infty$ with $m$, we can conclude that $h:=\sum_{m\ge 1} (a_m a^{-1}) f_p^m$ is an element of $M_0(O_K,r)$. Since $\frac{\Es_k}{F} = 1 + a \cdot h$ with $v_p(a)>0$, this means that $\frac{\Es_k}{F}$ is invertible in $M_0(O_K,r)$. Thus, if we consider the expansion $\frac{F}{\Es_k} = 1 + \sum_{m\ge 1} b_m f_p^m$ we must have
$$
v_p(b_m) \ge \frac{12}{p-1} \cdot v_p(r) m = \frac{12}{p-1} \cdot u m
$$
for all $m\ge 1$. As we can choose $u$ arbitrarily close to $\frac{t}{t+1} \cdot p\rho$, we must have $v_p(b_m) \ge \frac{12}{p-1} \cdot \frac{t}{t+1} \cdot p\rho$ for all $m\ge 1$, and this means $\frac{F}{\Es_k} \in M_0(\Z_p, \ge \frac{t}{t+1} \cdot p\rho )$.

For $p\ge 5$ we can use the same arguments, working with the Katz expansion of $\frac{\Es_k}{F}$ instead of expansions in the uniformizer $f_p$.
\end{proof}

\begin{proof}[Proof of Theorem \ref{thm:main}] By Proposition \ref{prop:Es_vs_F} and the well-known properties of the Fro\-benius operator $V$, we find that $\frac{V(\Es_k)}{V(F)} = V\left( \frac{\Es_k}{F} \right)$ is invertible in $M_0(\Z_p,\ge \frac{t}{t+1} \cdot \rho)$. Combining this with Proposition \ref{prop:Es_vs_F} again, it follows that
$$
\frac{\Es_k}{V(\Es_k)} = \frac{\Es_k}{F} \cdot \left( \frac{V(\Es_k)}{V(F)} \right)^{-1} \cdot \frac{F}{V(F)} \in M_0(\Z_p,\ge \frac{t}{t+1} \cdot \rho) ,
$$
since by Corollary \ref{cor:oc_F/V(F)} we have $\frac{F}{V(F)} \in M_0(\Z_p,\rho) \subseteq M_0(\Z_p,\ge \frac{t}{t+1} \cdot \rho)$.
\end{proof}

\subsection{Some special phenomena for \texorpdfstring{$p=2,3$}{p=2,3}}\label{subsec:special_p=2,3} The statement of the following proposition is well-known by work of Emerton, Buzzard--Kilford, and Roe; for $p=2$, cf.\ \cite{emerton_thesis} and specificially \cite[Corollary 9(iii)]{buzzard_kilford}, for $p=3$ this is in essence \cite[Theorem 4.2]{roe_3-adic}.

\begin{prop}\label{prop:special_E/V(E)} Let $p\in \{2,3\}$, and let $k\in\N$ be even and $\ge 4$. Then
$$
\frac{\Es_k}{V(\Es_k)} \in 1 + p^s M_0\left(\Z_p,\ge \frac{1}{2p} \right)
$$
where
$$
s = v_p(k/2) .
$$
\end{prop}

Let us discuss briefly how this statement compares with Theorem \ref{thm:main}. Let us focus on the case where $p=2$: if we write $\frac{\Es_k}{V(\Es_k)} = 1 + \sum_{i\ge 1} a_i(k) f_2^i$ then Proposition \ref{prop:oc_vs_loeffler} tells us that the statement of Proposition \ref{prop:special_E/V(E)} translates to the statement that
$$
v_2(a_i(k)) \ge v_2(k/2) + 3i
$$
for all $i$ which is precisely what comes out of \cite[Corollary 9(iii)]{buzzard_kilford}. When $k\equiv 0 \pmod{4}$ the statement coming out of Theorem \ref{thm:main} is that
$$
v_2(a_i(k)) \ge 4 \cdot \frac{t}{t+1} \cdot i
$$
for all $i$ with $t$ as in the theorem. If also $k>4$ we have $t\ge 5$, and so the estimate from Theorem \ref{thm:main} is asymptotically better (w.r.t.\ $i$), though may be a bit worse for small $i$. However, when $k\equiv 2 \pmod{4}$ Theorem \ref{thm:main} is actually worse than Proposition \ref{prop:special_E/V(E)}. The same is true when $p=3$ and $3\nmid k$.

In other words, Theorem \ref{thm:main} is generally stronger than Proposition \ref{prop:special_E/V(E)} except in the cases $p=2,3$ and $k$ not divisible by $2p$.

This raises the question of how to view Proposition \ref{prop:special_E/V(E)} from the point of view of our uniform method of proof of Theorem \ref{thm:main}.

The purpose of this section is to answer this question. The answer we will give is that Proposition \ref{prop:special_E/V(E)} can still be obtained via our method, but we have to take some very special phenomena specifically for the primes $p=2,3$ into account (as do the works cited above.) The very special phenomena are related to stronger integrality properties of the $U$ operator when acting on overconvergent modular functions of sufficiently small rate of overconvergence:

\begin{prop}\label{prop:special_U} Let $p\in \{2,3\}$. Let $O_K$ be the ring of integers of a finite extension $K$ of $\Q_p$. Then the $U$ operator maps $M_0\left(O_K,\ge \frac{1}{2p} \right)$ to $M_0\left(O_K,\ge \frac{1}{2} \right)$.
\end{prop}

Again, this proposition is well-known, though perhaps not in our precise formulation; cf.\ \cite[Lemma 3.1]{emerton_thesis}, \cite[Corollary 3(iii)]{buzzard_kilford}, \cite[Corollary 2.5(3)]{roe_3-adic}. We will sketch our proof of the proposition below in subsection \ref{subsec:discussion_special_U}, but will first show how Proposition \ref{prop:special_E/V(E)} can be obtained from it by the method we have used. We start with the following lemma.

\begin{lem}\label{lem:special_F/V(F)} Let $p\in \{2,3\}$, let $k$ and $s$ be as in Proposition \ref{prop:special_E/V(E)}, and let $F$ be as in Definition \ref{dfn:F_etc}. Then
$$
\frac{F}{V(F)} \in 1 + p^s \cdot M_0\left(\Z_p,\ge \frac{1}{2p}\right) .
$$
\end{lem}

\begin{proof} In the cases where $k\not\equiv 0 \pmod{2p}$ the statement follows immediately from Proposition \ref{prop:rate_of_oc_E/VE}. Concerning the remaining two cases, let us consider the case where $p=3$ and $k\equiv 0 \pmod{6}$. The case where $p=2$ and $k\equiv 0 \pmod{4}$ is dealt with in a similar way.

Assume then $p=3$ and $k\equiv 0 \pmod{6}$. The expansion of the modular function $\frac{E_6}{V(E_6)}$ in terms of the uniformizer $f_p$ has form $\frac{E_6}{V(E_6)} = 1 + \sum_{i\ge 1} b_i f_p^i$. Since $\frac{E_6}{V(E_6)} \in M_0(\Z_p, \ge \frac{1}{p+1})$ by Proposition \ref{prop:rate_of_oc_E/VE}, we know that
$$
v_p(b_i) \ge \frac{12}{p-1} \cdot \frac{1}{p+1} \cdot i = \frac{3}{2} \cdot i
$$
for all $i\ge 1$. By enlarging the ring of coefficients to a ring $O_K$ containing an element $b$ of valuation $\frac{1}{2} = \frac{1}{p-1}$, we then see that we can write $\frac{E_6}{V(E_6)} = 1 + b \sum_{i\ge 1} b_i' f_p^i$ where $b_i'\in O_K$ satisfy $v_p(b_i') \ge \frac{3}{2}\dot i - \frac{1}{2} \ge i$ for all $i\ge 1$. This means that $g := \sum_{i\ge 1} b_i' f_p^i \in M_0(O_K,\ge \frac{p-1}{12}) = M_0(O_K,\ge \frac{1}{2p})$. Now, one finds for $m\in\N$ that $(1+bg)^m \in 1 + p^{v_p(m)} b \cdot M_0(O_K,\ge \frac{1}{2p})$ by easy induction on $v_p(m)$ using $v_p(b) = \frac{1}{p-1}$. We can then conclude that
$$
\frac{F}{V(F)} = \left( \frac{E_6}{V(E_6)} \right)^{k/6} \in 1+ p^{v_p(k)-1 + \frac{1}{p-1}} \cdot M_0\left(O_K,\frac{1}{2p}\right) .
$$

But this means that in the expansion $\frac{F}{V(F)} = 1 + \sum_{i\ge 1} c_i f_p^i$ we have $v_p(c_i) \ge v_p(k)-1 + \frac{1}{p-1} + i$ for all $i\ge 1$. But the $c_i$ are all in $\Z_p$, and so we must have $v_p(c_i) \ge v_p(k) + i$ for all $i\ge 1$ which means that $\frac{F}{V(F)} \in 1 + p^{v_p(k)} \cdot M_0\left(\Z_p,\ge \frac{1}{2p}\right)$, as desired.
\end{proof}

We can now sketch a proof of Proposition \ref{prop:special_E/V(E)} via the method we have used in this paper: first, given Lemma \ref{lem:special_F/V(F)} the arguments from the proof of Lemma \ref{lem:U^i(F)/F-U(F)/F} apply to show that $\frac{U^i(F)}{F} \in 1 + p^s M_0(\Z_p,\ge \frac{1}{2})$ for $i\in \N$: one shows this by induction on $i$ where the base case $i=1$ results from applying $U$ to $\frac{F}{V(F)}$ and using Lemma \ref{lem:special_F/V(F)} combined with Proposition \ref{prop:special_U}. The induction step comes about exactly as in the proof of Lemma \ref{lem:U^i(F)/F-U(F)/F}, starting with the assumption $\frac{U^i(F)}{F} - 1 \in p^s M_0(\Z_p,\ge \frac{1}{2}) \subseteq p^s M_0(\Z_p,\ge \frac{1}{2p})$.

After this, as in the proofs of Proposition \ref{prop:Es_vs_F} and Theorem \ref{thm:main} one deduces first that $\frac{\Es_k}{F} \in 1 + p^s M_0(\Z_p,\ge \frac{1}{2})$, and then that Proposition \ref{prop:special_E/V(E)} holds.

\subsubsection{A discussion of Proposition \ref{prop:special_U}}\label{subsec:discussion_special_U} We will now sketch a proof of Proposition \ref{prop:special_U} that uses a uniform approach for both primes $p=2,3$. The purpose of this is two-fold: first, we want to understand better why the statement does not generalize in its precise form to any of the other primes $p$ for which $X_0(p)$ is of genus $0$ (that is, $p=5,7,13$.) Secondly, we wanted to obtain a deeper understanding of the reason why the proposition holds and whether something like it, i.e., that the $U$ operator preserves integral structure for overconvergent forms of sufficiently small rate of overconvergence, is true in some other situations.

We have not quite succeeded in these aims, as there are still elements of ``just so'' argumentation in our proof. But anyway, here is a sketch:

Fix $p$ to be any of the ``genus $0$ primes'' $p=2,3,5,7,13$ and write $f$ for the uniformizer $f_p$ of $X_0(p)$. For any integer $i\ge 0$ one can write
$$
U(f^i) = \sum_{j\ge 0} c_{i,j} f^j ,
$$
a finite sum with the $c_{i,j}$ in $\Z$. One has $c_{i,j} \neq 0 \Rightarrow \frac{i}{p} \le j \le ip$. For background on this, including how to use the modular equation to obtain recursion formulas for the coefficients $c_{i,j}$, one can refer to the proof of \cite[Lemma 3.3.2]{smithline_thesis}, or of \cite[Theorem 2.3]{loeffler}, or \cite[Chapter 2]{destefano} for a detailed account.

One has the following lower bound on the valuations of the $c_{i,j}$:
$$
v_p(c_{i,j}) \ge \gamma_p \cdot (pj-i) - 1
$$
where
$$
\gamma_p := \frac{12}{p^2-1} .
$$

This lower bound can be proved by brute computation by an induction where one uses the recursion formulas for the $c_{i,j}$, cf.\ \cite[Proposition 3.3.3]{smithline_thesis}. However, a more conceptual proof is also possible: consider a finite extension $K$ of $\Q_p$ with ring of integers $O_K$, and suppose that $r\in O_K$ with $\rho := v_p(r) < \frac{1}{p+1}$. Let $i\in \Z_{\ge 0}$. By \cite[Corollary 2.2]{loeffler} we then have $h:=r^{\frac{12}{p-1} i} f^i \in M_0(O_K,\rho)$. By the integrality properties of $U$ (cf.\ \cite[Integrality Lemma 3.11.4]{katz_padic}) we can then deduce that $pU(h) \in M_0(O_K,p\rho)$. But,
$$
pU(h) = pr^{\frac{12}{p-1} i} \sum_j c_{i,j} f^j = \sum_j pr^{\frac{12}{p-1} (i-pj)} c_{i,j} \left( r^{p\frac{12}{p-1}} f\right)^j ,
$$
and so, again by \cite[Corollary 2.2]{loeffler}, we must have
$$
0 \le v_p\left( pr^{\frac{12}{p-1} (i-pj)} c_{i,j} \right) = 1 + \frac{12}{p-1} \rho (i-pj) + v_p(c_{i,j}) .
$$

As we can choose a sequence of fields $K$ with elements $r$ such that $v_p(r)$ converges to $\frac{1}{p+1}$ from below, the claim follows.
\medskip

We now address the question of how to view Proposition \ref{prop:special_U} in this optic: Does there exist a positive number $\rho$ ($< \frac{1}{p+1}$) such that $U(M_0(O_K,\ge \rho) \subseteq M_0(O_K,\ge p\rho)$ whenever $K$ is a finite extension of $\Q_p$ with ring of integers $O_K$?

Using \cite[Corollary 2.2]{loeffler} again we can translate this question as asking whether there is such a positive $\rho$ such that $U(O_K[[rf]]) \subseteq O_K[[r^p f]]$ whenever $r\in O_K$ with $v_p(r) < \frac{12}{p-1} \rho$.

We have $U(O_K[[rf]]) \subseteq O_K[[r^p f]]$ if and only if for each fixed $i\ge 0$ we have $U(r^i f^i) = \sum_j b_j r^{pj} f^j$ with $b_j\in O_K$. Bringing in the above $c_{i,j}$ we can rewrite
$$
U(r^if^i) = \sum_j c_{i,j} r^{i-pj} \cdot r^{pj} f^j
$$
and so we see that there is an affirmative answer if and only if $v_p(c_{i,j}) \ge v_p(r) \cdot (pj-i)$ for all $i,j$. I.e., there is an affirmative answer if
$$
v_p(c_{i,j}) \ge a \cdot (pj-i)
$$
for all $i,j$ and some positive constant $a$ ($=\frac{12}{p-1} \rho$.)

Thus we find that the answer to the question is negative for $p=5,7,13$ as can be seen from special examples. For instance, if $p=5$ one computes $c_{4,1} = 24$ that has valuation $0$.

But the answer is affirmative for $p=2,3$: assume from now on that $p\in \{2,3\}$. We then claim that
$$
v_p(c_{i,j}) \ge \left( \gamma_p - \frac{1}{p-1} \right) \cdot (pj-i) \leqno{(\ast)}
$$
for all $i,j$.

This is now the point where we have to leave the conceptual approach and resort to ad hoc arguments: one can obtain a proof of $(\ast)$ by computing $c_{i,j}$ for $i<p$ and checking $(\ast)$, and then using the recurrence formula for the $c_{i,j}$ to prove the claim by induction on $i$ for $i\ge p$. Let us give a few details in case $p=2$: one then has $c_{1,1} = 2^3 \cdot 3$ and $c_{1,2} = 2^{11}$, and $(\ast)$ is true for $i=1$. For $i\ge 2$ one has the recurrence (convention: $c_{i,j} := 0$ if any of the indices is $<0$):
$$
c_{i,j} = 2^{12} c_{i-1,j-2} + 2^4\cdot 3\cdot c_{i-1,j-1} + c_{i-2,j-1}
$$
from which one immediately checks $(\ast)$ by induction on $i\ge 2$.

Given $(\ast)$ one can then conclude $U(M_0(O_K,\ge \rho)) \subseteq M_0(O_K,\ge p\rho)$ with
$$
\rho = \frac{p-1}{12} \cdot \left( \gamma_p - \frac{1}{p-1} \right) = \frac{1}{p+1} - \frac{1}{12} = \left\{ \begin{array}{ll} \frac{1}{4} ,& p=2\\ \frac{1}{6} ,& p=3 \end{array} \right\} = \frac{1}{2p}
$$
and so we have Proposition \ref{prop:special_U}.

\subsection{Some numerical examples}\label{subsec:numerical} Let us give a few numerical examples to illustrate the theorems.

Suppose first that $p\in \{2,3\}$ and $k\not\equiv 0 \pmod{2p}$. As we already explained above, if we expand $\frac{\Es_k}{V(\Es_k)} = 1 + \sum_{i\ge 1} a_i(k) f_p^i$ then Proposition \ref{prop:special_E/V(E)} translates into the statement that
$$
v_p(a_i(k)) \ge \frac{12}{p-1} \cdot \frac{1}{2p}\cdot i = \left\{ \begin{array}{ll} 3i & \mbox{ if } p=2 \\ i & \mbox{ if } p=3. \end{array} \right.
$$

As the expansion is easy to compute for any given $k$ (due to the fact that the $q$-expansion of $f_p$ starts with $q$), one easily finds examples showing that this cannot be improved in general: for $p=2$ consider $k=6$ or $k=10$. For $p=3$, consider for instance $k$ equal to $4$, $8$, or $10$.

If still $p\in\{ 2,3\}$, but $k\equiv 0 \pmod{2p}$, Theorem \ref{thm:main} implies instead that
$$
v_p(a_i(k)) \ge \frac{12}{p-1} \cdot \frac{t}{t+1} \cdot \frac{1}{p+1} \cdot i = \left\{ \begin{array}{ll} 4\cdot \frac{t}{t+1} i & \mbox{ if } p=2 \\ \frac{3}{2} \cdot \frac{t}{t+1} i & \mbox{ if } p=3 \end{array} \right.
$$
with $t$ as in the theorem. This is (w.r.t.\ $i$) asymptotically better than the statement resulting from Proposition \ref{prop:special_E/V(E)}.

As an example, take $p=2$, $k=12$: one computes $v_2(a_2(12)) = 7$, and as $t=5$ in this case, the lower bound from Theorem \ref{thm:main} is $\frac{20}{3}$. The lower bound from Proposition \ref{prop:special_E/V(E)} is exactly $7$ ($s=1$ in this case.) However, $v_2(a_{30}(12)) = 105$ with the lower bound from Theorem \ref{thm:main} equal to $100$ and the one from Proposition \ref{prop:special_E/V(E)} equal to $91$.

Finally, let us remark that Theorem \ref{thm:main} for primes $p\ge 5$ improves upon \cite[Theorem A]{kr} in case $v_p(k)>0$.

\end{document}